\documentclass[12pt]{article}
\usepackage{geometry}
\usepackage[utf8]{inputenc}
\usepackage{amsmath,amsthm,amssymb,color,graphicx,url,hyperref,diagbox,enumerate,cite,tikz,authblk}
\usepackage{appendix}

\newtheorem{theorem}{Theorem}

\newtheorem{problem}[theorem]{Problem}

\newtheorem{lemma}[theorem]{Lemma}

\newtheorem{corollary}[theorem]{Corollary}

\newtheorem{conjecture}[theorem]{Conjecture}

\newtheorem{definition}[theorem]{Definition}

\newcounter{encoding}

\newcommand{\fracpt}[1]{\left\{#1\right\}}
\newcommand{\N}{\mathbb{Z}^+}

\newcommand{\Z}{\mathbb{Z}}

\newcommand{\floor}[1]{\left\lfloor #1 \right\rfloor}

\title{Improved Ramsey-type theorems for Fibonacci numbers and other sequences}
\author{William J. Wesley}
\date{\today}

\begin{document}

\maketitle

\section{Introduction}
Van der Waerden's theorem \cite{VanDerWaerden_original} is a celebrated result in Ramsey theory. It states that for all positive integers $k$ and $r$, there exists a smallest integer $n = w(k,r)$, called the \emph{van der Waerden number}, such that every $r$-coloring of $[n]:=\{1,2,3,\dots,n\}$ contains a monochromatic $k$-term arithmetic progression. The problem of computing exact values of $w(k,r)$ is remarkably difficult: the largest known values are $w(6,2) = 1132$, $w(4,3) = 293$, and $w(3,4) = 76$ \cite{VDW26,VDW34,VDW43}. In this article we study two variants of van der Waerden numbers.

 A natural modification of the ``van der Waerden problem" is to impose the additional restriction that the common difference of the arithmetic progression must belong to a prescribed set $D$. We let $AP_D$ denote the set of all arithmetic progressions with common difference in $D$, and we let $n = n(AP_D,k;r)$ denote the smallest integer such that every $r$-coloring of $[n]$ contains a monochromatic $k$-term arithmetic progression whose common difference is in $D$. If $n(AP_D,k;r)$ exists for all $k$, then we say that $AP_D$ is \emph{$r$-regular}. If $AP_D$ is $r$-regular for all $r$, then we say that $AP_D$ is \emph{regular}. The largest $r$ such that $AP_D$ is regular is called the \emph{degree of regularity} of $AP_D$. 

Stated slightly differently, van der Waerden's theorem says that every $r$-coloring of $[w(k,r)]$ contains a monochromatic sequence $x_1 < x_2 < \dots < x_k$ where the differences between consecutive terms are all equal. Another modification of this problem is to ask for monochromatic sequences whose differences between consecutive terms are not necessarily equal, but all belong to some prescribed set $D \subset \Z^+$. We recall some terminology introduced in \cite{LandmanRobertson_SpecialGaps}. A sequence of integers $x_1 < \dots < x_k$ is called a $k$-term $D$-\emph{diffsequence} if $x_{i+1}-x_i \in D$ for $1 \le i \le k-1$. Unlike in van der Waerden's theorem, the existence of $D$-diffsequences is not guaranteed for all $D$. If $D$ is the set of odd positive integers, then the 2-coloring that colors the odd integers the first color and the even integers the second color contains no monochromatic $k$-term $D$-diffsequences, even for $k = 2$. We define $\Delta(D,k;r)$ to be the smallest number $n$ such that every $r$-coloring of $[n]$ contains a $k$-term monochromatic $D$-diffsequence. If $\Delta(D,k;r)$ exists for all $k$, then we say $D$ is \emph{$r$-accessible}. The largest integer $r$ for which $D$ is $r$-accessible is called the \emph{degree of accessibility} of $D$, denoted $doa(D)$. 

Landman and Robertson studied the existence of and bounds for the numbers $\Delta(D,k;r)$ for various choices of $D$, with particular emphasis on translates of the set of primes \cite{LandmanRobertson_SpecialGaps}. More recent work by Clifton \cite{CliftonDiffsequences} and Chokshi, Clifton, Landman, and Sawin \cite{RamseyFunctions_RestrictedGaps} has examined diffsequences involving sets such as 
$D = \{2^i : i\ge 0\}$ and given bounds on $\Delta(D,k;2)$. 

In this paper we are interested in the case when $D$ is the set of Fibonacci numbers $F = \{1,2,3,5,8,13,\dots\}$. Ramsey results involving sequences that satisfy the Fibonacci recurrence, among other linear recurrences, have been of general interest \cite{BertokNyul_MonochromeLinRecurrences,Landman_RamseyFunctions_Recurrences,NyulRauf_LinearRecurrence,HarborthMaasberg_RadoLinRecurrences,NyulRauf_LinRecurrenceConstantCoeff}. Landman and Robertson showed that $F$ is 2-accessible and left the matter of determining $doa(F)$ as an open question \cite{LandmanRobertson,LandmanRobertson_SpecialGaps}. In \cite{Fibonacci_Ramsey}, Ardal, Gunderson, Jungi\'c, Landman, and Williamson showed that $dor(AP_F)\le 5$ by constructing an explicit 6-coloring of $\Z^+$ that does not contain any monochromatic 2-term $F$-diffsequences. Moreover, they proved that $1 \le dor(AP_F) \le 3$ and gave several values of $\Delta(F,k;r)$. This paper builds upon the work in \cite{Fibonacci_Ramsey}: our main results are improvements on the bounds for $dor(AP_F)$ and $doa(F)$.
\begin{theorem}\label{Theorem_doa_F_le_3}
The degree of accessibility of the Fibonacci numbers $F$ is at most three. 
\end{theorem}

\begin{theorem}\label{Theorem_dor_AP_F}
The set $AP_F$ of arithmetic progressions whose gaps are Fibonacci numbers is not 2-regular. Moreover, $dor(AP_F) = 1$.  
\end{theorem}

Our proofs of these results involve combinatorial words that produce colorings that avoid either $F$-diffsequences or arithmetic progressions with gap in $F$ of a certain length. The search for these words was aided by the Online Encyclopedia of Integer Sequences (OEIS) and the computational power of SAT solvers, which have been used extensively to compute exact bounds for van der Waerden numbers and other related numbers such as Schur and Rado numbers \cite{VDW26,VDW34,WJW_Rado_ISSAC,BMRS_3ColorSchur,SchurFive}. We also use SAT solvers to compute other exact values of $\Delta(D,k;r)$ for other sets $D$. 

This paper is organized as follows. In Section \ref{Section_Preliminaries}, we formally define some of the objects studied in this paper and recall some well-known properties of Fibonacci numbers. Section \ref{Section_MainProofs} contains the proofs of Theorems \ref{Theorem_doa_F_le_3} and \ref{Theorem_dor_AP_F}. We conclude in Section \ref{Section_Experiments} with some experimental data and additional questions. 
\section{Preliminaries}\label{Section_Preliminaries}

In this section we collect several results that are used in the proofs of Theorems \ref{Theorem_doa_F_le_3} and \ref{Theorem_dor_AP_F}. We also fix the following notation for numbers and objects used throughout this paper.

We let $F := \{1,2,3,5,8,\dots\}$ denote the set of Fibonacci numbers, and let $G := \{1,4,17,72,\dots\} = \{ \frac f 2 : f \in F\} \cap \Z $. We let $f_i$ be the $i$-th term of the Fibonacci sequence, where $f_1 = f_2 = 1$ and $f_{n+1} = f_n +f_{n-1}$ for $n > 2$. Similarly, we let $g_i = \frac{f_{3i}}2$. We denote the Lucas numbers $\ell_n$ by $\ell_0 = 2, \ell_1 = 1$, $\ell_n = \ell_{n-1}+\ell_{n-2}$ for $n \ge 2$. For any real number $r$, we denote the fractional part of $r$ by $\{r\}:= r-\floor{r}$. We let $\phi$ denote the golden ratio $\phi := \frac{1+\sqrt{5}}2.$ When a number $\Delta(D,k;r)$ does not exist, we write $\Delta(D,k;r) = \infty$, and similarly for $n(AP_D,k;r)$.


The following lemma consists of two well-known results that give exact formulas for the Fibonacci numbers $f_n$ and Lucas numbers $\ell_n$ in terms of $\phi$. 

\begin{lemma}\label{Lemma_fn_closed_form}
The following identities for Fibonacci numbers $f_n$ and Lucas numbers $\ell_n$ hold.
\begin{enumerate}[(i)] 
  \item $f_n = \frac{\phi^n-(-\phi)^{-n}}{\sqrt{5}}$ for $n \ge 1$.
    \item  $\ell_n = \phi^n +(-\phi)^{-n}$ for $n \ge 0$.
\end{enumerate}
\end{lemma}
An immediate consequence of Lemma \ref{Lemma_fn_closed_form} is the following identity.
\begin{corollary}\label{Cor_divByPhi_all}
For $n \ge 2$, $\frac {f_n}{\phi} -f_{n-1} = (-1)^{n+1}\phi^{-n}.$
\end{corollary}




Recall that a \emph{word} is a (possibly infinite) sequence of symbols of a finite, nonempty alphabet. In this paper we consider only words over the alphabet $\{0,1\}$. The $n$-th \emph{Fibonacci word} $F_n$ is given by 
$$
F_0 = 0, \quad
F_1 = 01, \quad
F_n = F_{n-1}F_{n-2} \text{ if $n\ge 2$,}
$$
where $F_{n-1}F_{n-2}$ denotes the concatenation of $F_{n-1}$ and $F_{n-2}$. The \emph{infinite Fibonacci word} is the limit $F_\infty = 010010100100\dots$, the unique word that contains $F_n$ as a prefix for all $n$. We use the infinite Fibonacci word to define two new words, $S$ and $T$, which provide us colorings used in the proof of Theorems \ref{Theorem_doa_F_le_3} and \ref{Theorem_dor_AP_F}, respectively. 
\begin{definition}
Let $\mu$ be the word morphism given by $0 \mapsto 10$, $1\mapsto 01$, and let $\nu$ be the word morphism given by $0 \mapsto 1$, $1\mapsto 00$. The words $S$ and $T$ are given by $$S:=\mu(F_\infty)=1001101001\dots, \quad T:=\nu(F_\infty)=1001100100\dots.$$
\end{definition}

The morphism $\mu$ is known as the \emph{Thue-Morse morphism}. The fixed point of $\mu$ (which is unique up to binary complement) is the famous Thue-Morse infinite word $0110100110010\dots$, which has many interesting properties (see, for example, \cite{Christoffel_and_Repetitions}). 

The following lemma lists some key properties of the words $F_\infty, S$, and $T$. In particular, it shows that $F_\infty$ is \emph{Sturmian}. Sturmian words are well-studied and have several useful properties. In particular, the positions of the ones in a Sturmian word are given by terms in a \emph{Beatty sequence}, a sequence of the from $a_n = \floor{n\alpha}$ for some positive irrational $\alpha$. This property allows us to determine the positions of ones in $S$ and $T$ as well. We refer the interested reader to \cite{Christoffel_and_Repetitions} for additional results on Sturmian words. 
\begin{lemma}\label{Lemma_fibWord_floor_phi}
\begin{enumerate}[(i)]
    
\item 
The infinite Fibonacci word $F_\infty$ satisfies $$F_\infty(n) = \begin{cases}
0 & \text{if } n = \floor{m\phi} \text{ for some integer } m, \\
1 & \text{otherwise.}
\end{cases}
$$
 
\item If $n \neq \floor{m\phi}$ for all integers $m$, then there exist integers $m'$ and $m''$ such that $n+1 = \floor{m' \phi}$ and $n-1 = \floor{m''\phi}$. 
 
\item The word $S$ satisfies 

$$S(n) = \begin{cases} 
0 & \text{if } n \text{ is even and } \frac n 2 = \floor{m\phi} \text{ for some integer } m,\\
1 & \text{if } n \text{ is even and } \frac n 2 \neq \floor{m\phi} \text{ for all integers } m, \\
0 & \text{if } n \text{ is odd and } \frac {n+1} 2 \neq \floor{m\phi} \text{ for all integers } m,\\
1 & \text{if } n \text{ is odd and } \frac {n+1} 2  = \floor{m\phi} \text{ for some integer } m.
\end{cases}
$$

\item  The word $T$ satisfies 

$$T(n) = \begin{cases}
1 & \text{if } n = 2\floor{m\phi}-m \text{ for some integer } m, \\
0 &\text{otherwise}. 
\end{cases}
$$
\end{enumerate}

\begin{proof}
The statement $(i)$ was proven in \cite{Stolarsky_Beatty_sequences}. The statement $(ii)$ follows easily from the fact that $1<\phi < 2$, and $(iii)$ follows from $(i)$ and the definition of $S$. The proof of $(iv)$ was originally given by Michel Dekking on the OEIS entry A287772 for $\nu F_\infty$. For completeness, we give a slightly different proof here. 

By $(i)$, the positions of the zeros in $F_\infty$ are given by the sequence $\floor{m\phi}$, and by the definition of $\nu$, every 1 in $T$ is obtained from a 0 in $F_\infty$. It therefore suffices to show that $\nu$ maps the 0 at position $\floor{m\phi}$ to the 1 at position $2\floor{m\phi} -m$. This is easy to verify for $m = 1$, and suppose it holds for $m = k$. 

We consider two cases, noting that $\floor{(k+1)\phi}-\floor{k\phi} \in \{ 1,2\}$ for all integers $k$.  First, if $\floor{(k+1)\phi} - \floor{k \phi} = 1$, then the $k$-th and $(k+1)$-th zeros in $F_\infty$ are adjacent, and so in $T$, the $k$-th and $(k+1)$-th ones are adjacent. By the induction hypothesis, the $k$-th one is in position $2\floor{k\phi}-k$ and the $(k+1)$-th one is in position 
$2\floor{k\phi} - k +1 = 2 \floor{(k+1)\phi}-(k+1)$ as desired. 

Now suppose $\floor{(k+1)\phi} - \floor{k\phi} = 2$. Then there is a one between the $k$-th and $(k+1)$-th zeros in $F_\infty$. Therefore in $T$, by the definition of $\nu$ and the induction hypothesis, the $(k+1)$-th one is in position $2 \floor{k\phi} -k + 3 = 2\floor{(k+1)\phi}-(k+1)$, which completes the proof. 
\end{proof}
\end{lemma}



\section{Proofs of Main Theorems}\label{Section_MainProofs}
In this section we prove our main results, Theorems \ref{Theorem_doa_F_le_3} and \ref{Theorem_dor_AP_F}. In each case we construct a coloring of $\N$ and show by contradiction that it does not contain a suitable monochromatic diffsequence or arithmetic progression.
Our main technique in the proofs of Theorems \ref{Theorem_doa_F_le_3} and \ref{Theorem_dor_AP_F} is constructing a sequence of numbers $\{m_i\phi\}$ whose fractional parts are either strictly increasing or strictly decreasing. If there exist $i$ and $j$ such that $|\{m_i\phi\} - \{m_j\phi\}|\ge 1$, then this is a contradiction. 

\subsection{Proof of Theorem \ref{Theorem_doa_F_le_3}}
To prove $F$ is not 4-accessible, we must find a 4-coloring of $\N$ with no $k$-term $F$-diffsequences for some positive integer $k$. Instead of working directly with such a 4-coloring, we will use a 2-coloring that avoids $k$-term $G$-diffsequences. The following lemma shows that the existence of this 2-coloring is enough to prove that $F$ is not 4-accessible. 
\begin{lemma}\label{Lemma_2Coloring_implies_4Coloring}
Suppose $G$ is not 2-accessible, i.e. $\Delta(G,k;2) = \infty$ for some $k$. Then $F$ is not 4-accessible.
\end{lemma}
\begin{proof}
Let $\chi: \N \to \{1,2\}$ be a 2-coloring of $\N$ that does not contain a monochromatic $k$-term $G$-diffsequence. Then define a 4-coloring $\chi': \N \to \{c_{1,1},c_{1,2},c_{2,1},c_{2,2}\}$ by 

$$\chi'(n) = \begin{cases} 
c_{1,\chi(\frac{n+1}{2})} & n \text{ odd}, \\
c_{2,\chi(\frac{n}{2})} & n \text{ even}.
\end{cases}
$$
Now suppose towards contradiction that $\chi'$ contains a $k$-term monochromatic $F$-diffsequence $n_1,\dots, n_k$. By the construction of $\chi'$, each term in the diffsequence has the same parity. Suppose first that $n_1,\dots, n_k$ are all odd. Then $n_{i+1} -n_i $ is even for $1 \le i \le k-1$. Moreover, observe that $\chi(\frac{n_1+1}{2}) = \dots = \chi(\frac{n_k+1}{2})$. Therefore $\frac{n_{i+1}+1}2 - \frac{n_{i}+1}{2} = \frac{n_{i+1} -n_i}2 \in G$ for $1 \le i \le k-1$, so $\frac{n_1+1}{2},\dots, \frac{n_k+1}{2}$ is a $k$-term $G$-diffsequence, a contradiction. If we assume instead that $n_1,\dots, n_k$ are even, then we can reach a contradiction by a similar argument, which completes the proof. 
\end{proof}

The following result gives several bounds on differences of fractional parts, and it is central to the proof of Theorem \ref{Theorem_doa_F_le_3}. 
\begin{lemma}\label{Lemma_master_lemma}
The following bounds hold. 
\begin{enumerate}[(i)]

    \item Suppose $x_2 - x_1 = \frac{g_i}{2}$ with $g_i \in G$ even. If $x_1 = \floor{m_1 \phi}$ and $x_2 = \floor{m_2\phi}$ for some integers $m_1$ and $m_2$, then $$\fracpt{m_2\phi}-\fracpt{m_1\phi} < \frac{ \phi^{-3i+1}-\phi} 4 < -0.38 .$$
    
    \item Suppose $x_2 - x_1 =  \frac{g_i+1}{2}$ with $g_i \in G$ odd. If $x_1 = \floor{m_1 \phi}$ and $x_2+1 = \floor{m_2\phi}$ for some integers $m_1$ and $m_2$, then $$\fracpt{m_2\phi}-\fracpt{m_1\phi} < \frac {3\phi-6} 4 <-0.28.$$
    
    \item Suppose $x_2 - x_1 = \frac{g_i-1}{2}$ with $g_i \in G$ odd and $i > 1$. If $x_1+1 = \floor{m_1 \phi}$ and $x_2 = \floor{m_2\phi}$ for some integers $m_1$ and $m_2$, then $$\fracpt{m_2\phi}-\fracpt{m_1\phi} < \frac {-5\phi +6}{4} < -0.52.$$

\end{enumerate}

\end{lemma}
\newcommand{\fmTwo}{\fracpt{m_2\phi}}
\newcommand{\fmOne}{\fracpt{m_1\phi}}

\begin{proof}
For $(i)$, we have $$ m_2\phi - m_1\phi -\fmTwo + \fmOne = x_2 - x_1 = \frac{g_i}{2} = \frac{f_{3i}}{4},$$ and after rearranging and applying Corollary \ref{Cor_divByPhi_all}, we have $$m_2 -m_1 = \frac{f_{3i}}{4\phi} +\frac{\fmTwo - \fmOne}{\phi} = \frac{f_{3i-1}}{4}-\frac{\phi^{-3i}}{4} +\frac{\fmTwo - \fmOne}{\phi}.$$ 

Since $g_i$ is even, it follows from the definition of $g_i$ that $i$ is even, so $f_{3i-1} \equiv 1 \pmod 4$. Observe also that $\left|\frac{\fmTwo - \fmOne}{\phi}\right| < 1$. Because $m_2 - m_1$ is an integer, it follows that either $\fmTwo - \fmOne = \phi(\frac{\phi^{-3i}-1}4)$ or $\fmTwo-\fmOne = \phi(1-\frac 14 + \frac{\phi^{-3i}}4)$. However, the latter case is impossible since we must have $|\fmTwo - \fmOne| < 1$, but $\phi(1-\frac 14 + \frac{\phi^{-3i}}4) >1$ for all $i$. The minimum value of $i$ is 2, so we have $$\fmTwo-\fmOne = \frac{\phi^{-3i+1}-\phi}{4} \le  \frac{\phi^{-5}-\phi}{4} < -0.38$$ for all $i$. 

The proofs of the remaining parts are similar. For $(ii)$, since $g_i$ is odd, $g_i = \frac{f_{3i}}2$ with $i$ odd. Then we have 

$$m_2\phi - m_1\phi -\fmTwo + \fmOne = x_2 + 1 - x_1 = \frac{g_i+3}{2} = \frac{f_{3i}}{4} + \frac 32 .$$

Rearranging and applying Corollary \ref{Cor_divByPhi_all} gives 

$$m_2 - m_1 = \frac{f_{3i-1}}{4}+\frac{\phi^{-3i}}{4}+\frac{3}{2\phi}+\frac{\fmTwo - \fmOne}{\phi}.$$ 

Here, note that $f_{3i-1} \equiv 1 \pmod 4$ and $1<\frac{1}4 + \frac{3}{2\phi}<2$. By a similar argument as above, it follows that either $\fmTwo - \fmOne = -\phi \left(\frac{3}{2\phi} - \frac 34 + \frac{\phi^{-3i}}{4}\right) = \frac{3\phi-6}{4} -\frac{\phi^{-3i+1}}{4} $ or
$\fmTwo - \fmOne = \phi \left(1 -\left(\frac{3}{2\phi} - \frac 34 + \frac{\phi^{-3i}}{4}\right)\right) = \frac{7\phi-6}{4} -\frac{\phi^{-3i+1}}{4}$. Since $i \ge 1$, it follows that $\frac{7\phi-6}{4} -\frac{\phi^{-3i+1}}{4} >1$ and the latter case is impossible. Therefore $\fmTwo-\fmOne =\frac{3\phi-6}{4} -\frac{\phi^{-3i+1}}{4} < \frac{3\phi-6}{4} < -0.28$ for all $i$. 

For $(iii)$, we again have $g_i = \frac{f_{3i}}2$ with $i$ odd. Then we have 

$$m_2\phi - m_1\phi -\fmTwo + \fmOne = x_2 - (x_1+1) = \frac{g_i-3}{2} = \frac{f_{3i}}{4} - \frac 32. $$

Rearranging and applying Corollary \ref{Cor_divByPhi_all} gives 

$$m_2 - m_1 = \frac{f_{3i-1}}{4}+\frac{\phi^{-3i}}{4}-\frac{3}{2\phi}+\frac{\fmTwo - \fmOne}{\phi}.$$ 
Here we have $-1< \frac 14 - \frac{3}{2\phi} < 0$, so it follows that either $\fmTwo - \fmOne = -\phi \left( \frac 14 -\frac{3}{2\phi}+ \frac{\phi^{-3i}}{4}\right) = \frac{6-\phi}{4} - \frac{\phi^{-3i+1}}{4} > 1$ since $i > 1$, or  $\fmTwo - \fmOne = \phi \left(-1 -\left( \frac 14 -\frac{3}{2\phi}+ \frac{\phi^{-3i}}{4}\right)\right)  = \frac{6 -5\phi}{4} -\frac{\phi^{-3i+1}}{4}$. The former case is impossible, so $\fmTwo - \fmOne < \frac{6 -5\phi}{4}<-0.52$ for all $i$.   
\end{proof}

We are now equipped to prove Theorem \ref{Theorem_doa_F_le_3}. 

\begin{proof}[Proof of Theorem \ref{Theorem_doa_F_le_3}]
By Lemma \ref{Lemma_2Coloring_implies_4Coloring}, it is sufficient to find a 2-coloring of $\N$ that contains no monochromatic 4-term $G$-diffsequences. We will show that the coloring $\chi(n) = S(n),$ where $S(n)$ is the $n$-th symbol in the word $S$, satisfies this property. 

Throughout this proof, we assume that $y_1, y_2,y_3,y_4$ is a 4-term $G$-diffsequence, and we will show a contradiction. First, consider the following finite state machine. 
\begin{center}
\includegraphics[scale=0.35]{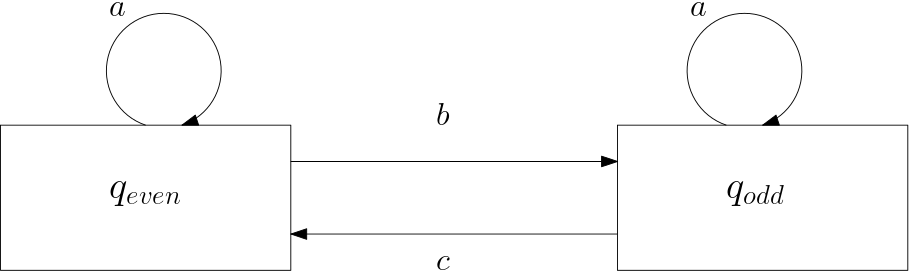}
\end{center}
Suppose $S(y_1) = S(y_2) = S(y_3) = S(y_4)$ with $y_{i+1}-y_i = g_{j_i}$ for some $g_{j_i} \in G$ and $i = 1,2,3$. For each diffsequence, we have a sequence of states $q_1,q_2,q_3,q_4$, and a sequence of transitions $t_1,t_2,t_3$.  If $y_i$ is even, then we set $q_i := q_{even}$, and if $y_i$ is odd, we set $q_i := q_{odd}$. The transitions $t_i$ are determined by the transition arrow that takes $q_i$ to $q_{i+1}$, so that, for example, if $q_1 = q_{even}$ and $q_2 = q_{odd}$, then $t_1 = b$. For each $y_i$, set $$x_i := \begin{cases} 
\frac{y_i}2 & \text{if }y_i \text{ is even }, \\
\frac{y_i+1}2 & \text{if }y_i \text{ is odd }.
\end{cases}$$  
By Lemma \ref{Lemma_fibWord_floor_phi}, for each $i$ there is a unique integer $m_i$ that satisfies $$ x_i = \begin{cases}
\floor{m_i \phi} &\text{ if }y_i \text{ is even and }S(y_i) = 0,\\
\floor{m_i \phi}-1 &\text{ if }y_i \text{ is odd and }S(y_i) = 0,\\
\floor{m_i \phi}+1 &\text{ if }y_i \text{ is even and }S(y_i) = 1,\\
\floor{m_i \phi} &\text{ if }y_i \text{ is odd and }S(y_i) = 1.\\
\end{cases} $$

We are done if we show that $\fracpt{m_4\phi} -\fracpt{m_1\phi} < -1$, which is a contradiction. By Lemma \ref{Lemma_master_lemma}, for all $i$ we have that 

$$ \fracpt{m_{i+1}\phi} - \fracpt{m_{i}\phi} <
\begin{cases}
-0.38 & \text{if } t_i = a, \\
-0.28 & \text{if } t_i = b, \\
-0.52 & \text{if } t_i = c. \\
\end{cases}$$
By examining all the combinations of values for $t_1, t_2,t_3$, we see that $\fracpt{m_4\phi} - \fracpt{m_1\phi} <-1$ unless $t_1 = t_2 = t_3 = b$. But this case is impossible since consecutive transitions cannot both be $b$, and the proof is complete.   

\end{proof}


\subsection{Proof of Theorem \ref{Theorem_dor_AP_F}}
The proof of Theorem \ref{Theorem_dor_AP_F} is similar to the proof of Theorem \ref{Theorem_doa_F_le_3}. We will show that the 2-coloring induced by $T = \nu(F_\infty)$ has no monochromatic 5-term arithmetic progressions whose gaps are in $F$. The following lemma is a technical result which is essential for the computations in the proof of Theorem \ref{Theorem_dor_AP_F}. 



\begin{lemma}\label{Lemma_Thm2_technical}
Let $f_n$ be the $n$-th Fibonacci number, and suppose $\epsilon \in \{-4,0,4\}$ and $n \ge 13$. Then the following identities hold: 

\begin{enumerate}[(i)]
\item
$$\fracpt{\frac{f_n+\epsilon}{2\phi-1}} =  \frac{2(-\phi)^{-n}}{5} + c_{n,\epsilon},$$ where
$$ c_{n,\epsilon} = 
\begin{cases}
\frac{4+(2\sqrt5-5)\epsilon}{10} & \text{if }   n\equiv 0 \pmod 4, \\
\frac{4+(4\sqrt 5-5)\epsilon}{20} & \text{if }n \equiv 1 \pmod 4  \text{ and } \epsilon \in \{0,4\}, \\
\frac{11-4\sqrt{5}}{5} & \text{if }  n \equiv 1 \pmod 4\text{ and } \epsilon = -4, \\
\frac{6+(2\sqrt{5}-5)\epsilon}{10} & \text{if } n \equiv 2 \pmod 4, \\
\frac{8 + (2\sqrt 5 - 5) \epsilon}{10} & \text{if } n \equiv 3 \pmod 4 \text{ and } \epsilon \in \{0,4\}, \\
\frac{9-4\sqrt{5}}{5} & \text{if }n\equiv 3 \pmod 4  \text{ and }   \epsilon = -4. 
\end{cases}$$
\item $\floor{\frac{f_n+\epsilon}{2\phi-1}}$ is even if and only if one of the following cases holds: 

\begin{itemize} \item$\epsilon = 0$ and $n \equiv 0,1,2,3,5,10 \pmod {12}$,

\item $\epsilon = 4$ and $n \equiv 0,2,3,9,10 \pmod {12}$, 
\item $\epsilon = -4$ and $n \equiv 0,1,2,5,7,10,11 \pmod {12}.$
\end{itemize} 
\end{enumerate}

\end{lemma}

\begin{proof}
By Lemma \ref{Lemma_fn_closed_form}, we have 

\begin{equation}\label{Eqn_f_Lucas_identity}\frac{f_n+ \epsilon}{2\phi-1} = \frac{f_n+\epsilon}{\sqrt{5}} = \frac{\phi^n - (-\phi)^{-n}}{5}+ \frac{\epsilon}{\sqrt{5}} = \frac{\ell_n}{5}-\frac{2(-\phi)^{-n}}{5}+ \frac{\epsilon}{\sqrt{5}}.
\end{equation}

Note that the Lucas numbers $\ell_n$ are periodic modulo 5 with period 4, so $\frac{\ell_n}{5} = m + \frac{r}{5}$ for some $m \in \Z$ and $r \in \{1,2,3,4\}$, with $r$ depending only on the value of $n$ modulo 4. Therefore, if $n$ is sufficiently large, then $\fracpt{\frac{f_n+\epsilon}{2\phi - 1}} = -\frac{2(-\phi)^{-n}}{5} + \fracpt{\frac{r}{5} + \frac{\epsilon}{\sqrt{5}}}$. Moreover, for $n \ge 13$, we have $\left|\frac{2(-\phi)^{-n}}{5}\right| < .001$, and some straightforward calculations give part $(i)$.

For part $(ii)$, we take the floor of both sides of Equation \ref{Eqn_f_Lucas_identity}. If we write $\ell_n = 10m +r$ for $m \in \Z$ with $0 \le r \le 9$, we observe that

$$\floor{\frac{f_n+\epsilon}{2\phi-1}} = \floor{\frac{10m+r}{5} -\frac{2(-\phi)^{-n}}{5}+\frac{\epsilon}{\sqrt{5}}} = 2k + \floor{\frac{r}5-\frac{2(-\phi)^{-n}}{5}+\frac{\epsilon}{\sqrt{5}}},$$
and we see that the parity of $\floor{\frac{f_n+\epsilon}{2\phi-1}}$ is dependent only on $\epsilon$ and the value of $r$. (The term $\frac{2(-\phi)^{-n}}{5}$ is negligible since $n \ge 13$ and the Lucas numbers are nonzero modulo 5.) The Lucas numbers are periodic modulo 10 with period 12, and a straightforward check of all the possible values of $\epsilon$ and $n \pmod {12}$ gives the result. 
\end{proof}

We now prove Theorem \ref{Theorem_dor_AP_F}. 
\begin{proof}[Proof of Theorem \ref{Theorem_dor_AP_F}]
Suppose towards contradiction that $T$ contains a 5-term arithmetic progression $x_1,x_2,x_3,x_4,x_5$ with common difference $f_n\in F$ and $T(x_1) = \dots = T(x_5)$. We first consider the case where $T(x_i) = 1$ for all $i$. By Lemma \ref{Lemma_fibWord_floor_phi}, for all $i$ there exists a positive integer $m_i$ such that $x_i = 2\floor{m_i \phi}-m_i$. 

Therefore for $i = 1,2,3,4$ we have 

$$m_{i+1}-m_i = 2(\floor{m_{i+1}\phi}-\floor{m_i\phi})+x_i-x_{i+1} = 2(\floor{m_{i+1}\phi}-\floor{m_i\phi})-f_n,$$ so $m_{i+1} -m_i+f$ is even, hence $m_{i+1}-m_i$ and $f_n$ have the same parity. After simplifying further, we obtain
$$m_{i+1}-m_i = 2 (m_{i+1}\phi - \fracpt{m_{i+1}\phi} -m_i\phi +\fracpt{m_i \phi})-f_n,$$ which implies 

$$m_{i+1}-m_i = \frac{f_n}{2\phi-1} +\frac{2}{2\phi-1}(\fracpt{m_{i+1}\phi} -\fracpt{m_i\phi}).$$

Since $\left|\frac{2}{2\phi-1}(\fracpt{m_{i+1}\phi} -\fracpt{m_i\phi})\right| <1,$ it follows that either $m_{i+1}-m_i = \lfloor \frac{f_n}{2\phi-1} \rfloor$ or
$m_{i+1}-m_i = \lceil\frac{f_n}{2\phi-1} \rceil$. By the parity argument above, $m_{i+1}-m_i$ is equal to the value in $\{\lfloor \frac{f_n}{2\phi-1} \rfloor ,\lceil\frac{f_n}{2\phi-1} \rceil\}$ that has the same parity as $f_n$. Therefore we have 

\begin{equation}\label{Eqn_fibAP_frac_diffs}
\{m_{i+1}\phi\} - \{ m_i \phi\} = \begin{cases} 
\frac{2\phi-1}2 \left(1-\fracpt{\frac {f_n}{2\phi-1}}\right) & \text{if } m_{i+1}-m_i = \lceil\frac{f_n}{2\phi-1} \rceil,
\\
\frac{2\phi-1}2 \left(-\fracpt{\frac {f_n}{2\phi-1}}\right) & \text{if } m_{i+1}-m_i = \lfloor\frac{f_n}{2\phi-1} \rfloor.
\end{cases}
\end{equation}

 The Fibonacci number $f_n$ is even if and only if $n$ is a multiple of 3, and so by using Equation \ref{Eqn_fibAP_frac_diffs} and Lemma \ref{Lemma_Thm2_technical}, we can now calculate the differences $\{m_{i+1}\phi\} - \{m_i\phi\}$. Note that these differences are dependent only on $f_n$, so they are equal for all $i$. If the absolute value of these differences is at least $\frac 14$, then there are no 5-term arithmetic progressions with $T(x_i) = 1$ for all $i$. We give the values of $\{m_{i+1}\phi\} - \{m_i\phi\}$, rounded to three decimal places, of $f_n$ below for $1\le n \le 12$: 
\begin{center}

\begin{tabular}{c|c|c}
     $n$ & $f_n$ & $\{m_{i+1}\phi\} - \{m_i\phi\}$\\
     \hline
    1,2 & 1& .618 \\
    3 & 2 & -1\\
    4 & 3 & -.382 \\
    5 & 5 & .854 \\
    6 &  8 & .472 \\
    7 &  13 & -.910 \\
    8 &  21 & -.438 \\
    9 & 34 & .889 \\
     10 &55 & .451 \\
     11 &89 & -.897 \\
     12 &144 & -.446 \\
     
\end{tabular}
\end{center}
For $n \ge 13$, using Equation \ref{Eqn_fibAP_frac_diffs} and Lemma \ref{Lemma_Thm2_technical}, we see that $\{m_{i+1}\phi\} - \{m_i\phi\}$ is approximately $\frac {-1}{\sqrt 5}$, $\frac 2{\sqrt 5}$, $\frac {1}{\sqrt 5}$, or $\frac {-2}{\sqrt 5}$ when $n$ is congruent to 0, 1, 2, or 3 modulo 4, respectively. We see that $| \{m_{i+1}\phi\} - \{m_i\phi\}| > \frac 13$ for all $n$, and so in fact there are not even any 4-term arithmetic progressions $x_1,x_2,x_3,x_4$ with gaps in $F$ with $T(x_i) = 1$ for $i=1,2,3,4$.

We now move to the case $T(x_i) = 0$. First, observe that the string 11 never appears in the word $F_\infty$, hence the string 000 never appears in $T$. Moreover, each 0 in $T$ is adjacent to another 0. Consequently, if $T(x) = 0$, then either $T(x-2) = 1$ or $T(x+2) = 1$. For each $x_i$, choose $y_i \in \{x_i-2,x_i+2\}$ such that $T(y_i) = 1$.  Therefore, if $x_{i+1}-x_i = f_n$ for all $i$, then  $y_{i+1}-y_i \in \{f_n-4,f_n,f_n+4\}$ for all $i$. By Lemma \ref{Lemma_fibWord_floor_phi}, for all $i$ there exists an $m_i$ such that $y_i = 2\floor{m_i \phi}-m_i$. 

Our analysis is now similar as above. Notice again that $m_{i+1}-m_i$ and $f_n$ must have the same parity since 

$$m_{i+1}-m_i = 2(\floor{m_{i+1}\phi}-\floor{m_i\phi})+y_i-y_{i+1} = 2(\floor{m_{i+1}\phi}-\floor{m_i\phi})-(f_n+\epsilon_i),$$ and $\epsilon_i \in \{-4,0,4\}$. 

Rearranging and using a similar argument as above, we have $m_{i+1}-m_i = \lfloor \frac{f+\epsilon_i}{2\phi-1} \rfloor$ or
$m_{i+1}-m_i = \lceil\frac{f+\epsilon_i}{2\phi-1} \rceil$, where $\epsilon_i \in \{-4,0,4\}$. Therefore we have 

\begin{equation}\label{Eqn_fibAP_frac_diffs_2}
\{m_{i+1}\phi\} - \{ m_i \phi\} = \begin{cases} 
\frac{2\phi-1}2 \left(1-\fracpt{\frac {f+\epsilon_i}{2\phi-1}}\right) & \text{if } m_{i+1}-m_i = \lceil\frac{f+\epsilon_i}{2\phi-1} \rceil,
\\
\frac{2\phi-1}2 \left(-\fracpt{\frac {f+\epsilon_i}{2\phi-1}}\right) & \text{if } m_{i+1}-m_i = \lfloor\frac{f+\epsilon_i}{2\phi-1} \rfloor.
\end{cases}
\end{equation}

We again calculate the first several values of $\{m_{i+1}\phi\}-\{m_i \phi\}$ when $\epsilon = \pm 4$ below.

\begin{center}
\begin{tabular}{c|c|c}
     $n$ &$f_n+4$& $\{m_{i+1}\phi\}-\{m_i \phi\}$  \\
     \hline
     1,2 &5 & .854 \\
     3 &6 & -.764 \\
     4 & 7 & -.146 \\
     5 & 9 & 1.090 \\
     6 & 12 & .708\\
     7 & 17 & -.674\\
     8 & 25 & -.202 \\
     9 & 38 & -1.111 \\
     10 & 59 & .687 \\
     11 & 93 & -.661 \\
     12 & 148 & -.210 \\
\end{tabular}
\hspace{40pt} 
\begin{tabular}{c|c|c}
    $n$ &  $f_n-4$& $\{m_{i+1}\phi\}-\{m_i \phi\}$ \\
     \hline
 1,2 & -3& .382\\
 3 & -2& 1\\
4 & -1& -.618\\
5 & 1& .618\\
6 & 4& .236\\
 7 & 9& 1.090\\
 8 & 17& -.674\\
9 & 30& .652\\
10 & 51& .215\\
11 & 85& 1.103 \\
 12 & 140& -.682\\
\end{tabular}
\end{center}

 If $\epsilon_i = 4$, then by Lemma \ref{Lemma_Thm2_technical}, for $n\ge 13$, the values of $\{m_{i+1}\phi\}-\{m_i \phi\}$ are approximately (within .001 of) one of the values $$ \frac{2(2-\sqrt5)}{\sqrt{5}}\approx -.211,\frac{2(1-\sqrt 5)}{\sqrt{5}} \approx -1.106,\frac{2(3-\sqrt5)}{\sqrt{5}} \approx .683,\frac{3-2\sqrt 5}{\sqrt 5}\approx -.658,$$ depending on whether $n \equiv 0,1,2$, or $3$ modulo 4, respectively. Similarly, if $\epsilon_i = -4$, then the values of  $\{m_{i+1}\phi\}-\{m_i \phi\}$ are approximately (within .001 of) one of the values $$\frac{-2(3-\sqrt5)}{\sqrt{5}} \approx -.683,\frac{-3+2\sqrt 5}{\sqrt 5}\approx .658,\frac{-2(2-\sqrt5)}{\sqrt{5}}\approx .211, \frac{-2(1-\sqrt 5)}{\sqrt{5}} \approx 1.106,$$ again depending on whether $n \equiv 0,1,2$, or $3$ modulo 4, respectively. 

Let $d_i:=\{m_{i+1}\phi\}-\{m_i\phi\}$. If $|d_i|\ge 1$ for any $i$, then we are done immediately. First, one can show that $d_i >0$ if and only if $\epsilon_i = 0$ and $n \equiv 1,2 \pmod 4$, or $\epsilon_i = 4$ and $n \equiv 2 \pmod 4$ or $n = 5$, or $\epsilon_i = -4$ and $n \not \equiv 0 \pmod 4$. 

Note also that $|d_i| \ge 1$ when  $\epsilon_i = 4$ and $n \equiv 1 \pmod 4$ and when $\epsilon_i = -4$ and $n \equiv 3 \pmod 4$. Therefore these two cases are impossible, and considering the remaining possibilities of $n \pmod 4$ and $\epsilon_i$, we see that $d_i$ always has the same sign regardless of $\epsilon_i$.

Observe that if $\epsilon_i = \pm 4$, then $\epsilon_{i+1} = 0$ or $\epsilon_{i+1} = \mp 4$. Using the approximations for $d_i$ for large $n$ and considering all possible sequences of $\epsilon_i$ for $i = 1,2,3,4$, we have $|d_1+d_2+d_3+d_4|\ge 1$ in all cases (in fact, $|d_1+d_2+d_3|\ge 1$ unless $n = 4$), which concludes the proof. 

\end{proof}

\section{Experimental Results and Further Questions}\label{Section_Experiments}

In this section we give some results on $\Delta(D,k;r)$ and $n(AP_D,k;r)$ for different sets $D$ and discuss some open questions. Our primary method for computing these values is with the SAT solver {\scshape{CaDiCaL}} \cite{BiereFazekasFleuryHeisinger-SAT-Competition-2020-solvers}. SAT solvers have been used successfully in arithmetic Ramsey theory to compute new values for van der Waerden numbers as well as the related Schur and Rado numbers \cite{VDW26,VDW34,BMRS_3ColorSchur,SchurFive,WJW_Rado_ISSAC}. 

\subsection{SAT solving}
 We first recall some standard SAT terminology, which can be found in, for example, \cite{SATHandbook_1stEd}. A \emph{literal} is a Boolean variable $v_i$ or its negation, $\bar{v}_i$. A \emph{clause} is a logical disjunction of literals. The formulas we use are in \emph{conjunctive normal form} (CNF), meaning they are conjunctions of clauses.

For each number $\Delta(D,k;r)$ (or $n(AP_D,k;r)$, we construct formulas $\phi_{n}$ that are satisfiable if and only if $\Delta(D,k;r) > n$ (or $n(AP_D,k;r)>n)$. The set of variables in each $\phi_n$ is $\{v_i^c : 1\le i\le n,\ 1\le c \le r\}$. The variable $v_i^c$ is assigned true if and only if integer $i$ is colored color $c$. The clauses we use in $\phi_n$ have three types, which, following \cite{SchurFive}, we call \emph{positive, negative}, and \emph{optional}.

Positive clauses ensure each integer $i$ is assigned at least one color, and are of the form 

$$ v_i^1 \vee v_i^2 \vee \dots \vee v_i^r$$ for $1\le i \le n$.

Negative clauses ensure that there are no monochromatic $D$-diffsequences (or arithmetic progressions with common difference in $D$). If $x_1,\dots, x_k$ is a $k$-term $D$-diffsequence (or arithmetic progression with common difference in $D$), then we include a clause of the form 

$$\bar{v}_{x_1}^c \vee \bar{v}_{x_2}^c \vee \dots \vee \bar{v}_{x_k}^c$$ for all colors $c$ and all $k$-term diffsequences (or arithmetic progressiona with common difference in $D$) with $1\le x_1 \le \dots \le x_k \le n$. 

Optional clauses ensure that each integer $i$ is assigned at most one color, and are of the form 

$$\bar{v}_{i}^c \vee \bar{v}_i^{c'}$$ for $1\le i \le n$ and $1\le c<c'\le r$. 

The formula $\phi_n$ is the conjunction of all positive, negative, and optional clauses for the given parameters $n,D,k,r$.

\subsection{Experimental results}
We first consider the set of Lucas numbers $L = \{2,1,3,4,7,11,\dots\}$ and $doa(L)$. It is easily shown that no Lucas number is a multiple of 5, so $\Delta(L,2;5) = \infty$ since coloring each integer its congruence class mod 5 avoids 2-term $L$-diffsequences. The following result gives a slight improvement and shows $doa(L)\le 3$.  

\begin{lemma}
Let $L$ be the set of Lucas numbers. Then $\Delta(L,3;4) = \infty$. In particular, $L$ is not 4-accessible and $doa(L)\le 3$. 
\end{lemma}
\begin{proof}
Define a coloring $\chi : \N \to [4]$ as follows: 
$$ \chi(n) = \begin{cases}
1 & \text{if } n \equiv 1,7 \pmod 8 ,\\
2 & \text{if } n \equiv 2,4 \pmod 8, \\
3 & \text{if } n \equiv 3,5 \pmod 8 ,\\
4 & \text{if } n \equiv 0,6 \pmod 8. 
\end{cases}
$$
We suppose towards contradiction that $x_1,x_2,x_3$ is a 3-term $L$-diffsequence with $\chi(x_1) = \chi(x_2) = \chi(x_3)$.
Observe that the Lucas numbers are periodic modulo 8 and congruent to $2,1,3,4,7,3,2,5,7,4,3,7,2,1,\dots \pmod 8$, so no Lucas number is congruent to $0$ or $6$ modulo 8. Thus by the definition of $\chi$, one can check that we must have $x_2-x_1 \equiv 2 \pmod 8$. But then we must have that $x_3 - x_2$ is congruent to either 0 or 6 modulo 8, which is a contradiction.
\end{proof}
Table \ref{Table_Lucas} gives some additional computed values of $\Delta(L,k;r)$. 

\begin{center} 
\begin{table}[hbt!]\caption{Table of numbers $\Delta(L,k;r)$}\label{Table_Lucas}
\centering
\bigskip

\begin{tabular}{c|cccccc}
\diagbox{$r$}{$k$}  & 2 & 3 & 4 & 5 & 6 & 7  \\
\hline
   2 & 3 & 5 & 7 & 13 & 15 & 21\\
   3 & 4 & 13 & 22 & 51\\
   4 & 5 & $\infty$ \\
   5 & $\infty$
\end{tabular}

\end{table}
\end{center}

We next study the set $P = \{2,3,5,7,10,12,17,22,\dots\}$, the set of nonzero Perrin (or ``skiponacci") numbers $p_n$, which are given by $p_1 = 3$, $p_2 = 0$, $p_3 = 2$, and $p_n = p_{n-2}+p_{n-3}$ for $n \ge 4$. Table \ref{Table_Perrin} gives some values of $\Delta(P,k;r)$. 
\begin{table}[hbt!]
\centering
\caption{Table of numbers $\Delta(P,k;r)$}
\label{Table_Perrin}
\bigskip
\begin{tabular}{c|cccccc}
\diagbox{$r$}{$k$}  & 2 & 3 & 4 & 5 & 6 & 7  \\
\hline
   2 & 5 & 9 & 13 & 19 & 23 & 31\\
   3 & 7 & 17 & 28 & 43\\
   4 & 13 & 35 & 81\\
   5 & 18 & 107\\
   6 & 25 \\
   7 & $>2000$\\
\end{tabular}
\end{table}
The most difficult computation was the upper bound $\Delta(P,3;5)\le 107$, which required over 5 hours using {\scshape CaDiCaL}.
\subsection{Open questions}
The proofs of Theorem \ref{Theorem_doa_F_le_3} and Theorem \ref{Theorem_dor_AP_F} give new bounds on $doa(F)$ and $doa(AP_F)$ by showing, respectively, that $\Delta(F,4;4) =\infty$ and $n(AP_F,5;2) =\infty$. It is known from \cite{Fibonacci_Ramsey} that $\Delta(F,2;4) = 9$, and a SAT solver easily shows $n(AP_F,3;2) = 17$. However, we were unable to compute the values $\Delta(F,3;4)$ and $n(AP_F,4;2)$. 

Using a greedy algorithm, we were able to find a 2-coloring of $[50000]$ that does not contain any 3-term $G$-diffsequences, which implies the bound $\Delta(G,3;2) > 50000$. The coloring used in the proof of Lemma \ref{Lemma_2Coloring_implies_4Coloring} then gives the bound $\Delta(F,3;2) > 100000$. Moreover, with a SAT solver we were able to show $n(AP_F,4;2) > 8000$. Given the large gaps between $\Delta(F,3;4)$ and $\Delta(F,4;4)$ as well as $n(AP_F,3;2)$ and $n(AP_F,4;2)$, we feel there is sufficient evidence to make the following conjecture. 

\begin{conjecture}
$\Delta(F,3;4) = n(AP_F,4;2) = \infty$.
\end{conjecture}
\section*{Acknowledgements} The author would like to thank Jes\'us De Loera for helpful comments on this paper and Clark Kimberling for uploading many relevant OEIS entries. This project was supported by National Science Foundation grant DMS-1818969.
\clearpage
\bibliographystyle{abbrv}
\bibliography{Bibliography.bib}

\end{document}